\documentclass[12pt]{article}
\usepackage{amssymb,latexsym,theorem,hyperref}
\usepackage{amsmath}

\newenvironment{proof}{\paragraph{Proof}
}{\qed\\}

\usepackage{graphicx}
\usepackage[all]{xy}
\usepackage{array}

\newcommand{\Spec}{\mathop{\mathrm{Spec}}\nolimits}

\newcommand{\red}{{\mathop{\mathrm{red}}\nolimits}}

\newcommand{\hull}{\mathop{\mathrm{hull}_\nabla}\nolimits}
\newcommand{\even}{{\mathrm{even}}}

\newcommand{\gr}{\mathop{\mathrm{gr}}\nolimits}

\newcommand{\id}{\mathop{\mathrm{id}}\nolimits}
\newcommand{\Si}{\mathfrak{S}}

\newcommand{\GL}{\mathop{\mathit{GL}}\nolimits}
\newcommand{\gl}{\mathop{\mathfrak{gl}}\nolimits}

\newcommand{\SL}{\mathop{\mathit{SL}}\nolimits}

\newcommand{\Hom}{\mathop{\mathrm{Hom}}\nolimits}

\newcommand{\Ga}{{\mathbb G_a}}

\newcommand{\A}{{\cal A}}

\newcommand{\qed}{\unskip\nobreak\hfill\hbox{ $\Box$}}

\newtheorem{Theorem}{Theorem}[section]

\newtheorem{Lemma}[Theorem]{Lemma}
\newtheorem{Corollary}[Theorem]{Corollary}
\theorembodyfont{\normalfont}
\newtheorem{Remark}[Theorem]{Remark}

\newtheorem{Definition}[Theorem]{Definition}
\newtheorem{Notation}[Theorem]{Notation}
\newtheorem{NotationConvention}[Theorem]{Notations and sign conventions}

\title{Bifunctor cohomology and Cohomological finite generation for reductive groups}
\author{Antoine Touz\'e \& Wilberd van der Kallen}
\begin{document}

\maketitle
\sloppy
\begin{abstract}
Let $G$ be a reductive linear algebraic group over a field $k$.
Let $A$ be a finitely generated commutative $k$-algebra on which $G$ acts
rationally by $k$-algebra automorphisms. Invariant theory tells
that the ring of invariants $A^G=H^0(G,A)$
is finitely generated.
We show that in fact the full cohomology ring $H^*(G,A)$ is finitely generated.
The proof is based on the strict polynomial
bifunctor cohomology classes constructed in \cite{Touze}. 
We also continue the study of bifunctor cohomology of $\Gamma^*(\gl^{(1)})$. 
\end{abstract}

\section{Introduction}
Consider  a linear algebraic  group $G$, or linear algebraic group scheme $G$,
defined over a field $k$.
So $G$ is an affine group scheme whose coordinate algebra $k[G]$ is finitely
generated as a $k$-algebra.
We say that $G$ has the cohomological finite generation property (CFG)
if the following holds.
Let $A$ be a finitely generated commutative $k$-algebra on which $G$ acts
 rationally by $k$-algebra automorphisms. (So $G$ acts from the right on $\Spec(A)$.)
Then the cohomology ring $H^*(G,A)$ is finitely generated
as a $k$-algebra. 

Here, as in
\cite[I.4]{Jantzen}, we use the cohomology introduced by
Hochschild, also known as `rational cohomology'.

Our main result confirms a conjecture of the senior author:
\begin{Theorem}\label{reductiveCFG}
Any reductive linear algebraic group over $k$ has property (CFG).
\end{Theorem}
The proof will be based on the `lifted' universal classes  \cite{Touze} constructed  by the junior author
for this purpose. Originally \cite{Touze} was the end of the proof, but for the purpose of exposition we
have changed the order.

If the field $k$ has characteristic zero, then the theorem just reiterates a standard fact in invariant
theory. Indeed the reductive group is then linearly reductive and rational
cohomology vanishes in higher degrees for any
linearly reductive group.

So we further assume $k$ has positive characteristic $p$. 
In this introduction we will also take $k$ algebraically closed. 
(One easily reduces to this case, cf.\ \cite[Lemma 2.3]{reductive}, \cite[I 4.13]{Jantzen}, \cite{Waterhouse}.)
We will say that $G$ acts on the algebra
$A$ if
$G$ acts
 rationally by $k$-algebra automorphisms. 

Let us say that $G$ has the finite generation property (FG),
or a positive solution to Hilbert's 14-th problem, if the following holds.
If  $G$ acts on a finitely generated commutative $k$-algebra $A$, then 
  the  ring of invariants $A^G=H^0(G,A)$ is finitely generated
as a $k$-algebra. Observe that, unlike Hilbert, we do \emph{not} require that $A$ is a domain.

It is obvious that (CFG) implies (FG). We will see that our main result can also be formulated as follows

\begin{Theorem}\label{FGCFG}
A linear algebraic  group scheme
$G$ over $k$ has property (CFG) if and only if it has property (FG).
\end{Theorem}

Let us give some examples. The first example is a finite group $G$, viewed as a discrete algebraic group over $k$.
It is well known to have property (FG), \cite[Lemma 2.4]{reductive}, and the proof goes back to
Emmy Noether 1926 \cite{noether}. 
Thus we recover the finite generation theorem of Evens, at least over our field $k$:
\begin{Theorem}[Evens 1961 \cite{Evens}]\label{CFGEvens}
A finite group has property (CFG), over $k$.
\end{Theorem}
As our proof of Theorem \ref{FGCFG} does not rely on theorem \ref{CFGEvens}, we get a new proof of \ref{CFGEvens}, 
albeit much longer than the original proof.
Note that the setting of Evens is more general: Instead of a field
he allows an arbitrary noetherian base. This suggests a direction for further work.

If $G$ is a linear algebraic group over $k$, we write $G_r$ for its $r$-th Frobenius kernel,
the scheme theoretic kernel of the $r$-th iterate $F^r:G\to G^{(r)}$ of the Frobenius homomorphism \cite[I Ch. 9]{Jantzen}.
It is easy to see that $G_r$ has property (FG). More generally it is easy to see \cite[Lemma 2.4]{reductive}
that any finite group scheme over $k$ has
property (FG). (A group scheme is finite if its coordinate ring is a finite dimensional vector space.)
Indeed one has \cite[Theorem 3.5]{reductive}:
\begin{Theorem}[Friedlander and Suslin 1997 \cite{Friedlander-Suslin}]
\label{CFGFS}
A finite group scheme has property (CFG).
\end{Theorem}
But we do not get a new proof of this theorem, as our proof of the main result relies heavily on the
specific information in \cite[section 1]{Friedlander-Suslin}. Recall that the theorem of 
Friedlander and Suslin
was motivated
by a desire to get a theory of support varieties for infinitesimal group schemes. 
Our problem has the same origin. Then it started to get  a life of its own and became a conjecture.

It is a theorem of Nagata \cite{Nagata}, \cite[Ch. 2]{Springer} 
that geometrically reductive groups (or group schemes \cite{Borsari-Santos}) have property (FG).
(Springer \cite{Springer} deletes `geometrically' in the terminology.)
Conversely, it is elementary \cite[Th. 2.2]{reductive} that property (FG) implies geometric reductivity.
(Here it is essential that in propery (FG) one allows any finitely generated commutative $k$-algebra on which $G$ acts.)
So our main result states that property (CFG) is equivalent to geometric reductivity.

Now Haboush has shown \cite[II.10.7]{Jantzen} that reductive groups are geometrically reductive, and Popov \cite{PopovRed}
has shown that 
a linear algebraic group with property (FG) is reductive.  (Popov allows only reduced
algebras, so his result is even stronger.) 
Waterhouse has completed the classification by showing \cite{Waterhouse} that  
a linear algebraic group scheme $G$ (he calls it an `algebraic affine group scheme') is geometrically
reductive exactly when the connected component $G_\red^o$ of
its reduced subgroup $G_\red$ is reductive. 
So this is also a characterization of the $G$ with property (CFG).

Let us now give a consequence of (CFG).
 We say that $G$ acts on an $A$-module $M$
 when it acts rationally on $M$ such that the structure map $A\otimes M\to M$ is a $G$-module map.
\begin{Theorem}Let $G$ have property (CFG).
Let $G$ act on the finitely generated commutative $k$-algebra $A$ and on the noetherian $A$-module $M$.
Then $H^*(G,M)$ is a noetherian $H^*(G,A)$-module. In particular, if $G$ is reductive and $A$ has a good filtration, then
$H^*(G,M)$ is a noetherian $A^G$-module, $H^i(G,M)$ vanishes for large $i$, and $M$ has finite good filtration
dimension.
\end{Theorem}
\paragraph{Proof}See \cite[Lemma 3.3, proof of 4.7]{reductive}.
 One puts an algebra structure on $A\oplus M$ and uses that $A\otimes k[G/U]$ also has a good filtration. 
\qed

As special case we mention
\begin{Theorem}\label{SvdKthm}Let $G=\GL_n$, $n\geq1$.
Let $G$ act on the finitely generated commutative $k$-algebra $A$ and on the noetherian $A$-module $M$.
If $A$ has a good filtration, then
$H^*(G,M)$ is a noetherian $A^G$-module, $H^i(G,M)$ vanishes for large $i$, and $M$ has finite good filtration
dimension.
\end{Theorem}

This theorem is proved directly in \cite{Srinivas vdK}, with functorial resolution of the ideal of the diagonal in
a product of Grassmannians. It will be used in our proof of the main theorems.

Now let us start discussing the proof of the main result.
First of all one has the following variation on the ancient transfer principle \cite[Chapter Two]{Grosshans book}.

\def\citemylemma{\cite[Lemma 3.7]{reductive}}
\begin{Lemma}[\citemylemma]\label{mylemma}
Let $G$ be a linear algebraic group over $k$ with property (CFG).
Then any geometrically reductive subgroup scheme $H$ of $G$ also
has property (CFG).
\end{Lemma}

As every geometrically reductive linear algebraic group scheme is a subgroup scheme of $\GL_n$ for $n$ sufficiently
large, we only have to look at the $\GL_n$ to prove the main theorems, Theorem \ref{reductiveCFG} and Theorem \ref{FGCFG}.
Therefore we further assume $G=\GL_n$ with $n>1$. (Or $n\geq p$ if you wish.)
(In \cite{cohGrosshans} we used $\SL_n$ instead of $G=\GL_n$, but also explained it hardly makes any difference.)

We have $G$ act on $A$ and wish to show $H^*(G,A)$ is finitely generated. If $A$ has a good filtration \cite{Jantzen}, then there is
no higher cohomology and invariant theory (Haboush) does the job.
A general $A$ has been related to one with a good filtration by Grosshans.
He defines a filtration $A_{\leq 0}\subseteq A_{\leq 1}\cdots$
on $A$ and embeds the associated graded $\gr A$ into an algebra with a good
filtration $\hull\gr A$. He shows that $\gr A$ and $\hull\gr A$ are also finitely generated and that there is a flat
family parametrized by the affine line with special fiber $\gr A$ and general fiber $A$.
We write $\A$ for the coordinate ring of the family. It is a graded algebra and one has natural homomorphisms
$\A\to \gr A$, $\A\to A$.
Mathieu has shown \cite{Mathieu G}, cf.\ \cite[Lemma 2.3]{cohGrosshans}, that there is an $r>0$ so that $x^{p^r}\in \gr A$
for every $x\in \hull\gr A$. We have no bound on $r$, which is the main reason that our results are only qualitative.
One sees that $\gr A$ is a noetherian module over the $r$-th Frobenius twist $(\hull\gr A)^{(r)}$
of $\hull\gr A$. So we do not quite have the situation of theorem \ref{SvdKthm}, but it is close.
We have to untwist. Untwisting involves $G^{(r)}=G/G_r$ and we end up looking at the Hochschild--Serre 
spectral sequence
$$E_2^{ij}=H^i(G/G_r,H^j(G_r,\gr A))\Rightarrow H^{i+j}(G,\gr A).$$
One may write $H^i(G/G_r,H^j(G_r,\gr A))$ also as $H^i(G,H^j(G_r,\gr A)^{(-r)})$.
By Friedlander and Suslin $H^*(G_r,\gr A)^{(-r)}$ is a noetherian module over the graded algebra
$\bigotimes_{i=1}^rS^*((\gl_n)^\#(2p^{i-1}))\otimes \hull\gr A$.
Here the $\#$ refers to taking a dual, $S^*$ refers to a symmetric algebra over $k$,
and the $(2p^{i-1})$ indicates in what degree one puts a copy of
the dual of the adjoint representation  $\gl_n$. By the fundamental work \cite{Akin} 
of Akin, Buchsbaum, Weyman, which is also of essential importance in \cite{Srinivas vdK}, one knows that 
$\bigotimes_{i=1}^rS^*((\gl_n)^\#(2p^{i-1}))\otimes \hull\gr A$ has a good filtration.
So $H^*(G_r,\gr A)^{(-r)}$ has finite good filtration dimension and page 2 of our Hochschild--Serre
spectral sequence is noetherian over its first column $E^{0*}_2$.
By Friedlander and Suslin  $H^*(G_r,\gr A)^{(-r)}$ is a finitely generated algebra
and by invariant theory $E^{0*}_2$
is thus finitely generated, so $E^{**}_2$ is finitely generated. The spectral sequence is one of graded commutative
differential graded algebras
in characteristic $p$, so the $p$-th power of an even cochain in a page passes to the next page. It easily follows
that all pages are finitely generated. As page 2 has only finitely many columns by \ref{SvdKthm}, cf.~\cite[2.3]{Srinivas vdK}, 
this explains why the abutment
$H^*(G,\gr A)$ is finitely generated. We are getting closer to $H^*(G,A)$.

The filtration $A_{\leq 0}\subseteq A_{\leq 1}\cdots$ induces a filtration
of the Hochschild complex \cite[I.4.14]{Jantzen} whence a spectral sequence
$$E(A):E_1^{ij}=H^{i+j}(G,\gr_{-i}A)\Rightarrow H^{i+j}(G, A).$$
It lives in the second quadrant, but as 
$E^{**}_1$ is a finitely generated $k$-algebra this causes no difficulty with
convergence: given $m$ there will be only finitely many nonzero $E_1^{m-i,i}$. 
(Compare \cite[4.11]{cohGrosshans}. Note that in \cite{cohGrosshans}
the $E_1$ page is mistaken for an $E_2$ page.) All pages are again finitely generated, so we would like the spectral sequence to
stop, meaning that $E_s^{**}=E_\infty^{**}$ for some finite~$s$. There is a standard method to achieve this \cite{Evens}, \cite{Friedlander-Suslin}. One must find a `ring of operators' acting on the spectral sequence and show that some page is 
a noetherian module for the chosen ring of operators. As the ring of operators we take $H^*(G,\A)$.
Indeed $E(A)$ is acted on by the trivial 
spectral sequence $E(\A)$ whose pages equal $H^*(G,\A)$, see \cite[4.11]{cohGrosshans}. And $H^*(G,\A)$
also acts on our Hochschild--Serre spectral sequence through its abutment. 
If we can show that one of the pages of the Hochschild--Serre  spectral sequence
is a noetherian module over $H^*(G,\A)$, then that will do the trick, as then the abutment  $H^*(G,\gr A)$ is noetherian
by \cite[Lemma 1.6]{Friedlander-Suslin}. And this abutment is the first page of $E(A)$.

Now we are in a situation similar to the one encountered by Friedlander and Suslin. Their problem was `surprisingly
elusive'. To make their breakthrough they had to invent strict polynomial functors. Studying the homological algebra of
 strict polynomial functors they found universal cohomology classes $e_r\in H^{2p^{r-1}}(G, \gl_n^{(r)})$
 with nontrivial restriction to $G_1$. That was enough to get through.
 We faced a similar bottleneck. We know from invariant theory and from \cite{Friedlander-Suslin} that page 2 of our   
 Hochschild--Serre 
 spectral sequence is noetherian over $H^0(G,(\bigotimes_{i=1}^rS^*((\gl_n)^\#(2p^{i-1}))\otimes \A)^{(r)})$.
 But we want it to be noetherian over $H^*(G,\A)$. So if we could factor the homomorphism
 $H^0(G,(\bigotimes_{i=1}^rS^*((\gl_n)^\#(2p^{i-1}))\otimes \A)^{(r)})\to E^{0*}_2$ through $H^*(G,\A)$, then that would do it.
 The universal classes $e_j$ provide such a factorization on some summands, but they do not seem to help on the rest.
 One would like to have universal classes in more degrees so that one can map every
 summand of the form
 $H^0(G,(\bigotimes_{i=1}^rS^{m_i}((\gl_n)^\#(2p^{i-1}))\otimes \A)^{(r)})$ into the appropriate $H^{2m}(G,\A)$, or even into
  $H^{2m}(G,\A^{G_r})$.
 The dual of $S^{m_i}((\gl_n)^\#)^{(r)}$ is $\Gamma^{m_i} (\gl_n^{(r)})$. 
 Thus one seeks nontrivial classes in $H^{2mp^{i-1}}(G,\Gamma^m (\gl_n^{(r)}))$, to take cup product with. It turns out that
 $r=i=1$ is the crucial case and we seek nontrivial classes $c[m]\in H^{2m}(G,\Gamma^m (\gl_n^{(1)}))$.
 The construction of such classes $c[m]$ has been a sticking point at least since 2001.
 In \cite{cohGrosshans} they were constructed for $\GL_2$, but one needs them for $\GL_n$ with $n$ large.
 The strict polynomial functors of Friedlander and Suslin do not provide a natural home for this problem, but the
 strict polynomial bifunctors \cite{Franjou-Friedlander} of Franjou and Friedlander do.
 
 When the junior author found \cite{Touze}
 a construction of nontrivial `lifted' classes $c[m]$, this finished a proof of the conjecture.
 We present two proofs. The first proof continues the investigation of bifunctor cohomology in \cite{Touze}
 and establishes properties of the $c[m]$ analogous to those employed in \cite{cohGrosshans}.
 Then the result follows as in the proof in \cite{cohGrosshans} for $\GL_2$. 
 As a byproduct one also obtains extra bifunctor cohomology classes and relations between them.
 The second proof needs no more properties of the classes $c[m]$ than those established in \cite{Touze}.
 Indeed \cite{Touze} stops exactly where the two arguments start to diverge.
 The second proof does not quite factor the homomorphism
 $H^0(G,(\bigotimes_{i=1}^rS^*((\gl_n)^\#(2p^{i-1}))\otimes \A)^{(r)})\to E^{0*}_2$ through $H^*(G,\A)$,
 but argues by induction on $r$, returning to \cite[section 1]{Friedlander-Suslin} with the new classes in hand.
It is not hard to guess which author contributes which proof. The junior author goes first.

\section{Acknowledgements}The senior author thanks MATPYL, ``Math\'ematiques des Pays de Loire'', 
for supporting a visit to  Nantes during most of January 2008.
There he stressed that any construction of non-zero $c[m]$ regardless
of their properties should be useful -- and this has indeed turned out to be the case.
\part{The first proof}
\section{Main theorem and cohomological finite generation}\label{three}

We work over a field $k$ of positive characteristic $p$. We keep the notations
of \cite{Touze}. In particular, $\mathcal{P}_k(1,1)$ denotes the category of
strict polynomial bifunctors of \cite{Franjou-Friedlander}.
The main result of part I is theorem \ref{thm-cl-univ}, which states the existence of
classes in the cohomology of the bifunctor $\Gamma^*(gl^{(1)})$. By
\cite[Thm 1.3]{Touze}, the cohomology of a bifunctor $B$ is related to the
cohomology of $GL_{n,k}$ with coefficients in the rational representation 
$B(k^n,k^n)$ by a map $\phi_{B,n}:H^*_\mathcal{P}(B)\to H^*(GL_{n,k},B(k^n,k^n))$
(natural in $B$ and compatible with cup products). So our main result yields
classes in the cohomology of $GL_{n,k}$, actually more classes 
(and more relations between them) than originally needed 
\cite[Section 4.3]{cohGrosshans} 
for the proof of 
the cohomological finite generation conjecture.

\begin{Theorem}\label{thm-cl-univ}
Let $k$ be a field of characteristic $p>0$.
There are maps $\psi_\ell: \Gamma^\ell H^*_{\mathcal{P}}(gl^{(1)})\to H^*_{\mathcal{P}}(\Gamma^\ell(gl^{(1)}))$, $\ell\ge 1$ 
such that
\begin{enumerate}
\item $\psi_1$ is the identity map.
\item For all $\ell\ge 1$ and for all $n\ge p$,  the composite 
$$\Gamma^\ell H^*_{\mathcal{P}}(gl^{(1)})\xrightarrow[]{\psi_\ell} H^*_{\mathcal{P}}(\Gamma^\ell(gl^{(1)}))\xrightarrow[]{\phi_{\Gamma^\ell(gl^{(1)}),n}} 
H^*(GL_{n,k},\Gamma^\ell(\mathfrak{gl}_n^{(1)}))$$
is injective. In particular, for all $\ell\ge 1$, $\psi_\ell$ is injective.
\item For all positive integers $\ell,m$, there are commutative diagrams
$$\xymatrix{
H^*_{\mathcal{P}}(\Gamma^{\ell+m}(gl^{(1)}))\ar@{->}[rr]^-{\Delta_{\ell,m\,*}}&& H^*_{\mathcal{P}}(\Gamma^{\ell}(gl^{(1)})\otimes\Gamma^m(gl^{(1)}))\\
\Gamma^{\ell+m}H^*_{\mathcal{P}}(gl^{(1)})\ar@{^{(}->}[rr]^-{\Delta_{\ell,m}}\ar@{->}[u]^{\psi_{\ell+m}}&& 
\Gamma^{\ell}H^*_{\mathcal{P}}(gl^{(1)})\otimes\Gamma^m H^*_{\mathcal{P}}(gl^{(1)})\ar@{->}[u]^{\psi_{\ell}\cup\psi_{m}}\;,
}$$
and
$$\xymatrix{
H^*_{\mathcal{P}}(\Gamma^{\ell}(gl^{(1)})\otimes\Gamma^m(gl^{(1)}))\ar@{->}[rr]^-{m_{\ell,m\,*}}&& H^*_{\mathcal{P}}(\Gamma^{\ell+m}(gl^{(1)}))\\
\Gamma^{\ell}H^*_{\mathcal{P}}(gl^{(1)})\otimes\Gamma^m H^*_{\mathcal{P}}(gl^{(1)})\ar@{->}[u]^{\psi_{\ell}\cup\psi_{m}}\ar@{->}[rr]^-{m_{\ell,m}}&&
\Gamma^{\ell+m}H^*_{\mathcal{P}}(gl^{(1)})\ar@{->}[u]^{\psi_{\ell+m}}\;,
}$$
where $m_{\ell,m}$ and $\Delta_{\ell,m}$ denote the maps induced by the multiplication 
$\Gamma^\ell\otimes\Gamma^m \to \Gamma^{\ell+m}$ and the diagonal $\Gamma^{\ell+m}\to \Gamma^\ell\otimes\Gamma^m$.
\end{enumerate} 
\end{Theorem}

As a consequence, we obtain that \cite[Th 4.4]{cohGrosshans} is valid for any value of $n$:

\begin{Corollary}\label{cor-univ-class} Let $k$ be a field of positive characteristic. For all $n>1$, there are classes 
$c[m]\in H^{2m}(GL_{n,k},\Gamma^m(\mathfrak{gl}_n^{(1)}))$ such that
\begin{enumerate}
\item  $c[1]$ is the Witt vector class $e_1$,
\item  $\Delta_{i,j\,*}(c[i+j])=c[i]\cup c[j]$ for $i,j\ge 1$.
\end{enumerate}
\end{Corollary}

\begin{proof} Arguing as in \cite[Lemma 1.5]{Touze} we notice that it suffices to prove the statement when $n\ge p$. 
By \cite[Th 1.3]{Touze}, we have morphisms 
$$\phi_{\Gamma^m(gl^{(1)}),n}: H^*_\mathcal{P}(\Gamma^m(gl^{(1)}))\to H^*(GL_{n,k},\Gamma^m(\mathfrak{gl}_n^{(1)}))$$
compatible with the cup products, and for $m=1$ the map $\phi_{\Gamma^m(gl^{(1)}),n}$ is an isomorphism.
Let $b[1]$ be the pre-image of the Witt vector class by $\phi_{\Gamma^1(gl^{(1)}),n}$. We define 
$c[m]:=(\phi_{\Gamma^m(gl^{(1)}),n}\circ \psi_m)(b[1]^{\otimes m})$.
Then $c[1]$ is the Witt vector class since $\psi_1$ is the identity, and by theorem \ref{thm-cl-univ}(3) the classes $c[i]$ 
satisfy condition 2.
\end{proof}

\begin{Corollary}
The cohomological finite generation conjecture (Theorem \ref{reductiveCFG}) 
holds.
\end{Corollary}
\begin{proof}
Let $G$ be a reductive linear algebraic group acting on a finitely generated commutative
$k$-algebra $A$. We want to prove that $H^*(G,A)$ is finitely generated. 
To do this, it suffices to follow \cite{cohGrosshans} and this is exactly what we do
below. We keep the notations of the introduction. 

By lemma
\ref{mylemma}, the case $G=GL_{n,k}$ suffices. As recalled in the 
introduction, there exists a positive integer $r$ such that  
the Hochschild-Serre
spectral sequence
$$E_2^{ij}=H^i(G/G_r,H^j(G_r,\gr A))\Rightarrow H^{i+j}(G,\gr A)$$
stops for a finite good filtration dimension reason. Moreover it
is a sequence of finitely generated algebras, and its second page is
noetherian over its subalgebra $E_2^{0*}$ (all this was first proved in 
\cite[Prop 3.8]{cohGrosshans}, under some restrictions on the characteristic
which were removed in \cite{Srinivas vdK}).

The composite $\A^{G_r}\hookrightarrow \A \twoheadrightarrow \gr A$ makes
$\gr A$ into a noetherian module over $\A^{G_r}$. Hence, by \cite[Thm
1.5]{Friedlander-Suslin}
(with $\text{`$C$'}=\A^{G_r}$) and by invariant theory \cite[Thm 16.9]{Grosshans book},
$E_2^{0*}=H^0(G/G_r,H^*(G_r,\gr A))$ (hence $E_2^{**}$) is noetherian over 
$H^0(G/G_r,\bigotimes_{i=1}^r
S^*((\gl_n^{(r)})^\#(2p^{i-1}))\otimes\A^{G_r})$. 

Now we use the classes of corollary \ref{cor-univ-class} as in section 4.5 and in the
proof of corollary 4.8 of \cite{cohGrosshans}. In this way, we factor 
the morphism 
$H^0(G/G_r,\bigotimes_{i=1}^r
S^*((\gl_n^{(r)})^\#(2p^{i-1}))\otimes\A^{G_r})\to E_2^{0*}$ through
the map $H^\even(G,\A)\to H^0(G/G_r,H^\even(G_r,\A))= E_2^{0\,\even}$
(the latter map is induced by restricting the cohomology from $G$ to $G_r$).
So $E_2^{**}$ is noetherian over $H^\even(G,\A)$. By 
\cite[Lemma 1.6]{Friedlander-Suslin} (with $\text{`$A$'}=H^\even(G,\A)$ and 
$\text{`$B$'}=k$), we
conclude that the map $H^\even(G,\A)\to H^*(G,\gr A)$ (induced by $\A\to \gr
A$) makes $H^*(G,\gr A)$ into a noetherian module over $H^\even(G,\A)$.

The proof finishes as described in the introduction (or in section 4.11 of 
\cite{cohGrosshans}): the second spectral sequence
$$E(A):E_1^{ij}=H^{i+j}(G,\gr_{-i}A)\Rightarrow H^{i+j}(G, A)$$
is a sequence of finitely generated algebras. It is acted on by the trivial 
spectral sequence $E(\A)$ whose pages equal $H^*(G,\A)$. But we have proved
that $E_1^{**}$ is noetherian over $H^*(G,\A)$, so by the usual trick
(\cite{Evens}, \cite{Friedlander-Suslin} or \cite[Lemma 3.9]{reductive}) the
spectral sequence $E(A)$ stops, which proves that  $H^*(G,A)$ is finitely
generated.
\end{proof}

\section{Proof of theorem \ref{thm-cl-univ}}\label{four}

By \cite[Prop 3.21]{Touze}, the divided powers $\Gamma^\ell$ admit a twist
compatible coresolution $J_\ell$. So by \cite[Prop 3.18]{Touze}, we have a 
bicomplex $A(J_\ell)$ whose totalization yield an $H^*_\mathcal{P}$-acyclic
coresolution of $\Gamma^\ell(gl^{(1)})$. In particular the homology of 
the totalization of $H^0_\mathcal{P}(A(J_\ell))$ computes
$H^*_\mathcal{P}(\Gamma^\ell(gl^{(1)}))$. 

The plan of the proof of theorem \ref{thm-cl-univ} is the following. 
First, we build the maps $\psi_\ell$. To be more specific, we build maps
$\vartheta_\ell$ which send each element of degree $d$ of   
$\Gamma^\ell(H^*_\mathcal{P}(gl^{(1)}))$ to a homogeneous cocycle of bidegree
$(0,d)$ in the bicomplex $H^0_\mathcal{P}(A(J_\ell))$. Our maps $\psi_\ell$ will
then be induced by the $\vartheta_\ell$.

Second, we show the relations
between the classes on the cochain level. In this step, 
we encounter the following problem: the cup product of two
classes is represented by  a cocycle 
in the bicomplex $H^0_\mathcal{P}(A(J_\ell)\otimes A(J_m))$
while we want to have it represented by a cocycle in 
$H^0_\mathcal{P}(A(J_\ell\otimes J_m))$. So we have investigate further the
compatibility of the functor $A$ with cup products.

Finally, we prove theorem \ref{thm-cl-univ}(2) by reducing to one parameter
subgroups.

\begin{NotationConvention}\label{notasgn}
If $\mathfrak{A}$ is an additive category, we denote by $Ch^{\ge
0}(\mathfrak{A})$ (resp. $\text{$p$-}Ch^{\ge
0}(\mathfrak{A})$, resp. $\text{bi-}Ch^{\ge
0}(\mathfrak{A})$) the category of nonnegative cochain complexes (resp.
$p$-complexes, resp. bicomplexes) in $\mathfrak{A}$.

If $\mathfrak{A}$ is equipped with a tensor product, then $Ch^{\ge
0}(\mathfrak{A})$ inherits a tensor product. The differential of the tensor
product $C\otimes D$ involves a Koszul sign: the
restriction of $d_{C\otimes D}$ to $C^i\otimes D^j$ equals
$d_C\otimes\mathrm{Id}+(-1)^i \mathrm{Id}\otimes d_D$.
The category $\text{$p$-}Ch^{\ge
0}(\mathfrak{A})$ also inherits a tensor product, but the $p$-differential of
$C\otimes D$ does not involve any sign: $d_{C\otimes
D}=d_C\otimes\mathrm{Id}+\mathrm{Id}\otimes d_D$.

Now we turn to bicomplexes. First, we may view a 
complex $C^\bullet$ whose terms $C^j$ are
chain complexes as a bicomplex $C^{\bullet,\bullet }$ 
whose object $C^{i,j}$ is the 
$i$-th object of the complex $C^j$ ({\it i.e.} the complexes $C^j$ are the rows of 
$C^{\bullet,\bullet }$). Thus we obtain an identification:
$$Ch^{\ge 0}(Ch^{\ge 0}(\mathfrak{A}))=\text{bi-}Ch^{\ge
0}(\mathfrak{A})\;.$$  
Being a category of cochain complexes, the term on the left hand side has a 
tensor product. If $C$ is a
bicomplex, let us denote by $d_C^{i,j}:C^{i,j}\to C^{i+1,j}$ its first
differential, and by $\partial_C^{i,j}:C^{i,j}\to C^{i,j+1}$ its second one.
Then one checks that the tensor product on bicomplexes induced by the
identification is such that the restriction of
$d_{C\otimes D}$ (resp. $\partial_{C\otimes D}$) to 
$C^{i_1,j_1}\otimes D^{i_2,j_2}$ equals
$d_C\otimes\mathrm{Id}+(-1)^{i_1}\mathrm{Id}\otimes d_D$ (resp. 
$\partial_C\otimes\mathrm{Id}+(-1)^{j_1}\mathrm{Id}\otimes \partial_D$).

We define the totalization $\mathrm{Tot}(C)$ of a bicomplex $C$
with the Koszul sign convention: the restriction of $d_{\mathrm{Tot}(C)}$ to
$C^{i,j}$ equals $d_C+(-1)^{i}\partial_C$. If $C$, $D$ are two bicomplexes,
there is a canonical isomorphism of complexes: $\mathrm{Tot}(C)\otimes
\mathrm{Tot}(D)\simeq \mathrm{Tot}(C\otimes D)$ which sends an element
$x\otimes y\in C^{i_1,j_1}\otimes D^{i_2,j_2}$ to $(-1)^{j_1i_2}x\otimes y$.
\end{NotationConvention}

\subsection{Construction of the $\psi_\ell$, $\ell\ge 1$}

Let $\ell$ be a positive integer. By \cite[Prop 3.18, Prop 3.21]{Touze} we
have a bicomplex $H^0_\mathcal{P}(A(J_\ell))$ whose homology computes the
cohomology of the bifunctor $\Gamma^\ell(gl^{(1)})$. We now recall the 
description of the first two columns of this bicomplex. 
As in \cite[Section 4]{Touze}, we denote
by $A_1$ the $p$-coresolution of $gl^{(1)}$ 
obtained by precomposing the $p$-complex $T(S^1)$ by the bifunctor $gl$. The
symmetric group $\Si_\ell$ acts on the $p$-complex $A_1^{\otimes \ell}$ by permuting the
factors of the tensor product (unlike the case of ordinary complexes, the
action of $\Si_\ell$ does not involve a Koszul sign since the tensor product 
of $p$-complexes does not involve any sign).
Contracting the $p$-complex $A_1^{\otimes \ell}$ and applying
$H^0_\mathcal{P}$, we obtain an action of $\Si_\ell$ on the ordinary complex
$H^0_\mathcal{P}((A_1^{\otimes\ell})_{[1]})$.
By \cite[Lemma 4.2]{Touze}, the first two columns 
$H^*_\mathcal{P}(A(J_\ell)^{0,\bullet})\to 
H^*_\mathcal{P}(A(J_\ell)^{1,\bullet})$ of $H^0_\mathcal{P}(A(J_\ell))$
equal 
$$ \underbrace{H^0_\mathcal{P}\left((A_1^{\otimes \ell})_{[1]}\right)}_{
\text{column of index $0$}}\xrightarrow[]{\,\prod 
(1-\tau_i)\,}
\underbrace{\bigoplus_{i=0}^{\ell-2}H^0_\mathcal{P}\left((A_1^{\otimes \ell}
)_{[1]}\right)}_{\text{column of index $1$}}\;,$$
where $\tau_i\in\Si_\ell$ is the transposition which exchanges $i+1$ and $i+2$ (and with the convention that the second column is 
null if $\ell=1$). Thus we have:
\begin{Lemma}\label{lm-cocycles-col-1}
Let $\mathcal{Z}_\ell^\even$ be the set of homogeneous cocycles of bidegree $(0,d)$,
$d$ even, in the bicomplex $H^0_\mathcal{P}(A(J_\ell))$. Then
$\mathcal{Z}_\ell^\even$ identifies as the set of even degree cocycles of the
complex $H^0_\mathcal{P}((A_1^{\otimes\ell})_{[1]})$, which are invariant
under the action of $\Si_\ell$.
\end{Lemma}

Now we turn to building a map $\vartheta_\ell:\Gamma^\ell(H^*_\mathcal{P}(gl^{(1)}))\to
\mathcal{Z}_\ell^\even$. In view of lemma \ref{lm-cocycles-col-1}, it suffices to
build a $\Si_\ell$-equivariant map 
$\vartheta_\ell:H^*_\mathcal{P}(gl^{(1)})^{\otimes\ell}
\to H^0_\mathcal{P}((A_1^{\otimes\ell})_{[1]})$. 

Let us first recall what we know about $H^*_{\mathcal{P}}(gl^{(1)})$.
By \cite[Thm 1.5]{Franjou-Friedlander} and \cite[Th. 4.5]{Friedlander-Suslin}, the graded vector space 
$H^*_{\mathcal{P}}(gl^{(1)})$ is concentrated in degrees $2i$, $0\le i<p$ and one dimensional in these degrees. 
Following \cite{SFB}, we denote by $e_1(i)$ a generator of degree $2i$ of this graded vector space. 
The homology of the complex $H^0_\mathcal{P}(A_{1[1]})$ computes the
cohomology of the bifunctor $gl^{(1)}$.
Thus we may choose for each integer $i$, $0\le i<p$,  a cycle $z_i$ 
representing the cohomology class $e_1(i)$ in this complex.
The cycles $z_i$ determine a graded map 
$H^*_{\mathcal{P}}(gl^{(1)})\to H^0_\mathcal{P}(A_{1[1]})$. 
By \cite[Prop 3.3]{Touze}, we may take cup products on the cochain level 
to obtain for each $\ell\ge 1$ a map 
$$H^*_{\mathcal{P}}(gl^{(1)})^{\otimes \ell}\to H^0_\mathcal{P}(A_{1[1]})^{\otimes \ell}\xrightarrow[]{\cup} 
H^0_\mathcal{P}((A_{1[1]})^{\otimes \ell})\;.$$
Moreover we define chain maps $h_\ell:(A_{1[1]})^{\otimes \ell}\to
(A_1^{\otimes\ell})_{[1]}$ by iterated use of
\cite[Prop 2.7]{Touze}. More
specifically, $h_1$ is the identity and
$h_\ell=h_{A_1^{\otimes\ell-1},A_1}\circ (h_{\ell-1}\otimes h_1)$. 

\begin{Lemma}\label{lm-vartheta} Let $\ell$ be a positive integer and let 
$\vartheta_\ell$ be the composite
$$\vartheta_\ell:= H^*_{\mathcal{P}}(gl^{(1)})^{\otimes \ell}\to
H^0_\mathcal{P}((A_{1[1]})^{\otimes
\ell})\xrightarrow[]{\,H^0_\mathcal{P}(h_\ell)\,}
H^0_\mathcal{P}((A_{1}^{\otimes \ell})_{[1]})\;. $$
Then $\vartheta_\ell$ satisfies the following two properties:
\begin{enumerate}
\item[(1)]
The image of $\vartheta_\ell$ is contained in the set of even degree
cocycles of 
$H^0_\mathcal{P}((A_{1}^{\otimes \ell})_{[1]})$. 
\item[(2)] $\vartheta_\ell$
is $\Si_\ell$-equivariant.
\end{enumerate} 
\end{Lemma}
\begin{proof}
The first property is straightforward from the definition of
$\vartheta_\ell$. We prove the second one.
The map $H^*_{\mathcal{P}}(gl^{(1)})^{\otimes \ell}\to 
H^0_\mathcal{P}((A_{1[1]})^{\otimes \ell})$ is defined using cup 
products, hence it is $\Si_\ell$-equivariant. Thus, to prove the lemma, we have to study the map 
$h_\ell:(A_{1[1]})^{\otimes \ell}\to (A_1^{\otimes\ell})_{[1]}$. 

Recall that $h_\ell$ is built by iterated uses of \cite[Prop 2.7]{Touze}.  Thus, if we define the graded object 
$p(A_1,\dots,A_1)=\bigoplus_{i_1,\dots,i_\ell}\bigotimes_{s=1}^\ell A_1^{i_s p}$ with the component 
$\bigotimes_{s=1}^\ell A_1^{i_s p}$ in degree $2(\sum i_s)$, we have well defined inclusions of 
$p(A_1,\dots,A_1)$ into the complexes $(A_{1[1]})^{\otimes \ell}$ and $(A_1^{\otimes\ell})_{[1]}$. 
Moreover $h_\ell$ fits into a commutative diagram:
$$\xymatrix{
(A_{1[1]})^{\otimes \ell}\ar@{->}[rr]^-{h_\ell}&&(A_1^{\otimes\ell})_{[1]}\\
p(A_1,\dots,A_1)\ar@{^{(}->}[u]^{(a)}\ar@{=}[rr]&&p(A_1,\dots,A_1)\ar@{^{(}->}[u]^{(b)}\,.
}$$
Let $\Si_\ell$ act on $p(A_1,\dots,A_1)$ by permuting the factors of the 
tensor product, on $(A_{1[1]})^{\otimes \ell}$ 
by permuting the factors of the tensor product with a Koszul sign, 
and on $(A_1^{\otimes\ell})_{[1]}$ by 
permuting the factors of the tensor product $A_1^{\otimes\ell}$ (without
sign).  Then the map $(b)$ is equivariant, and the map 
$(a)$ is also equivariant since $p(A_1,\dots,A_1)$ is concentrated in even 
degrees. The map $h_\ell$ is \emph{not} equivariant. 
However, by definition the equivariant map $H^*_{\mathcal{P}}(gl^{(1)})^{\otimes \ell}\to 
H^0_\mathcal{P}((A_{1[1]})^{\otimes \ell})$ factors through
$H^0_\mathcal{P}(p(A_1,\dots,A_1))$  so that 
postcomposition of this map by $H^0_\mathcal{P}(h_\ell)$ 
(ie: the map $\vartheta_\ell$) is in fact equivariant.
\end{proof}

\begin{Notation}
By lemmas \ref{lm-cocycles-col-1} and \ref{lm-vartheta}, for all $\ell\ge 1$,
the map
$\vartheta_\ell$ induces a map $\Gamma^\ell(H^*_\mathcal{P}(gl^{(1)}))\to 
\mathcal{Z}^\even_\ell$. We denote by $\psi_\ell$ the composite
$$\psi_\ell:= \Gamma^\ell( H^*_{\mathcal{P}}(gl^{(1)})) \to
\mathcal{Z}_\ell^\even
\to H^*_\mathcal{P}(\Gamma^\ell(gl^{(1)})) \;. $$
\end{Notation}

\begin{Lemma}
The map $\psi_1$ equals the identity map. 
\end{Lemma}
\begin{proof}
For $\ell=1$, $\vartheta_1$ is just the map $H^*_{\mathcal{P}}(gl^{(1)})\to \Hom(\Gamma^{p}(gl),A_{1[1]})$ which sends the 
generator $e_1(i)$ of $H^{2i}_{\mathcal{P}}(gl^{(1)})$ to the cycle $z_i$ representing this generator. Moreover, by definition 
of $z_i$, the map $\mathcal{Z}_1^\even\twoheadrightarrow H^*_{\mathcal{P}}(gl^{(1)})$ sends $z_i$ to $e_1(i)$.  Thus, for all $i$, $\psi_1$ 
sends $e_1(i)$ to itself.
\end{proof}

\subsection{Proof of theorem \ref{thm-cl-univ}(3)}
Let $\mathcal{P}_k$ be the strict
polynomial functor category and let $\mathcal{TP}_k$ be the twist compatible
subcategory \cite[Def 3.9]{Touze}. Before proving theorem
\ref{thm-cl-univ}(3), we need to study further properties of the functor $A:Ch^{\ge
0}(\mathcal{TP}_k)\to \text{bi-}Ch^{\ge 0}(\mathcal{P}_k(1,1))$ \cite[Def.
3.17]{Touze}.
Recall that $A$ is defined as the composite of
the following three functors:
\begin{enumerate}
\item The `Troesch coresolution functor' \cite[Prop 3.13]{Touze}
$$ T: Ch^{\ge 0}(\mathcal{TP}_k)\to \text{$p$-}Ch^{\ge
0}(Ch^{\ge 0}(\mathcal{P}_k))\;,$$
\item The contraction functor 
$$ -_{[1]}: \text{$p$-}Ch^{\ge 0}(Ch^{\ge 0}(\mathcal{P}_k))\to 
Ch^{\ge 0}(Ch^{\ge 0}(\mathcal{P}_k))\;,$$
\item Precomposition by the bifunctor $gl$
$$ -\circ gl: Ch^{\ge 0}(Ch^{\ge 0}(\mathcal{P}_k))\to Ch^{\ge 0}(Ch^{\ge
0}(\mathcal{P}_k(1,1)))= \text{bi-}Ch^{\ge 0}(\mathcal{P}_k(1,1))\;.$$
\end{enumerate}
All the categories coming into play in the definition of $A$
are equipped with tensor products (cf. notations and sign conventions
\ref{notasgn}). The functors $T$ and $-\circ gl$ commute
with tensor products, but $-_{[1]}$ does not. As a result, if $F,G$ are
homogeneous strict
polynomial functors of respective degree $f,g$ and 
with respective twist compatible coresolutions $J_F,J_G$,
we have two (in general non isomorphic) $H^*_\mathcal{P}$-acyclic 
coresolutions of the tensor product $F\otimes G$ at our disposal:
$$\mathrm{Tot}(A(J_F))\otimes \mathrm{Tot}(A(J_G))\quad\text{and}
\quad\mathrm{Tot}(A(J_F\otimes J_G))\;.$$

Now the problem is the following. On the one hand,
cycles representing cup products of classes
in the cohomology of $F$ and $G$ are easily identified using 
the first complex.
Indeed, by \cite[Prop 3.3]{Touze}, the cup product
$$H^*_\mathcal{P}(F(gl^{(1)}))\otimes H^*_\mathcal{P}(G(gl^{(1)}))\to
H^*_\mathcal{P}(F(gl^{(1)})\otimes G(gl^{(1)})) $$ 
is defined at the cochain level by sending cocycles 
$x$ and $y$ respectively in 
$\Hom(\Gamma^{pf}(gl),\mathrm{Tot} (A(J_F)))$ and  
$\Hom(\Gamma^{pg}(gl),\mathrm{Tot} (A(J_G)))$ to the cocycle
$$x\cup y:= (x\otimes y)\circ \Delta_{pf,pg}\in
\Hom(\Gamma^{p(f+g)}(gl),\mathrm{Tot} (A(J_F))\otimes \mathrm{Tot} (A(J_G)))\,,$$
where $\Delta_{pf,pg}$ is the diagonal map $\Gamma^{p(f+g)}(gl)\to
\Gamma^{pf}(gl)\otimes \Gamma^{pg}(gl)$.
But on the other hand, by functoriality of $A$, if
$E\in\mathcal{P}_k$ then the effect of a
morphism $E\to F\otimes G$ is easily computed in
$H^0_\mathcal{P}(\mathrm{Tot}(A(J_F\otimes J_G)))$. So we want to be able to
identify cup products in $H^0_\mathcal{P}(\mathrm{Tot}(A(J_F\otimes J_G)))$
rather than in $H^0_\mathcal{P}(\mathrm{Tot}(A(J_F))\otimes
\mathrm{Tot}(A(J_G)))$. This is the purpose of next lemma.

\begin{Lemma}\label{lm-cup-prod-bicompl}
Let $F,G$ be homogeneous strict polynomial functors of degree $f$, $g$ which admit twist compatible coresolutions $J_F$ and $J_G$. 
Let 
$i,j,\ell,m$ be nonnegative integers, and let 
$$x_{i,2j}\in \Hom_{\mathcal{P}_{pf}^{pf}}(\Gamma^{pf} (gl),A(J_F))\text{ , }y_{\ell,2m}\in 
\Hom_{\mathcal{P}_{pg}^{pg}}(\Gamma^{pg} (gl), A(J_G))$$ 
be homogeneous cocycles of respective bidegrees $(i,2j)$ and $(\ell, 2m)$. 
\begin{enumerate}
\item The object $A(J_F)^{i,2j}\otimes  A(J_G)^{\ell,2m}$ appears once 
and only once in the bicomplex
$A(J_F\otimes J_G)$. It appears in bidegree
$(i+\ell, 2j+2m)$. In particular, the formula $$(x_{i,2j}\otimes y_{\ell,2m})\circ
\Delta_{pf,pg}\in \Hom(\Gamma^{p(f+g)}(gl),A(J_F)^{i,2j}\otimes 
A(J_G)^{\ell,2m} )$$
defines a homogeneous element of bidegree $(i+\ell, 2j+2m)$ in the bicomplex
$H^0_\mathcal{P}(A(J_F\otimes J_G))$.
\item The element $(x_{i,2j}\otimes y_{\ell,2m})\circ
\Delta_{pf,pg}$ is actually a cocycle, and represents the cup product
$[x_{i,2j}]\cup[y_{\ell,2m}]$ in $H^0_\mathcal{P}(\mathrm{Tot}(A(J_F\otimes
J_G)))$. 
\end{enumerate}
\end{Lemma}

\begin{proof}
Since $T$ commutes with tensor products \cite[Prop 3.13]{Touze}, the bicomplex
$A(J_F\otimes J_G)$ is naturally isomorphic to the precomposition by $gl$ of
the bicomplex $(T(J_F)\otimes T(J_G))_{[1]}$, while $A(J_F)\otimes A(J_G)$
equals the precomposition by $gl$ of the bicomplex 
$T(J_F)_{[1]}\otimes T(J_G)_{[1]}$. Recall that in the identification of
$Ch^{\ge 0}(Ch^{\ge 0}(\mathcal{P}_k(1,1)))$ and 
$\text{bi-}Ch^{\ge 0}(\mathcal{P}_k(1,1))$, the $j$-th object of 
a complex of complexes 
$C^\bullet$ corresponds to the $j$-th row of the bicomplex 
$C^{\bullet,\bullet}$ (that is the elements of bidegree $(*,j)$).
So the first statement simply follows from
\cite[Lemma 2.2]{Touze}.
Furthermore, by \cite[Prop 2.4]{Touze} there is a map of bicomplexes
$$A(J_F)\otimes A(J_G)\to A(J_F\otimes J_G)$$ which is the identity on  
$A(J_F)^{i,2j}\otimes A(J_G)^{\ell,2m}$. 
Applying the functor $\mathrm{Tot}$ we obtain a map of $H^*_{\mathcal{P}}$-acyclic coresolutions
$$\theta\,:\,\mathrm{Tot}(A(J_F))\otimes \mathrm{Tot}(A(J_G))\simeq
 \mathrm{Tot}(A(J_F)\otimes A(J_G))\to 
\mathrm{Tot}(A(J_F\otimes J_G))$$ 
over the identity map of $F(gl^{(1)})\otimes G(gl^{(1)})$, and whose restriction to 
$A(J_F)^{i,2j}\otimes A(J_G)^{\ell,2m}$ equals the identity.
(more specifically, this equality holds up to a 
$(-1)^{2j\ell}=1$ sign coming from the sign in the formula 
$\mathrm{Tot}(C\otimes D)\simeq \mathrm{Tot}(C)\otimes\mathrm{Tot}(D)$).

By definition of the cup product, 
$(x_{i,2j}\otimes y_{\ell,2m})\circ\Delta_{pf,pg}$ is a cocycle representing 
$[x_{i,2j}]\cup[y_{\ell,2m}]$ in 
$H^0_\mathcal{P}(\mathrm{Tot}(A(J_F))\otimes \mathrm{Tot}(A(J_G)))$.  
Applying $\theta$ we obtain that 
$(x_{i,2j}\otimes y_{\ell,2m})\circ\Delta_{pf,pg}$ is a cocycle in 
$H^0_\mathcal{P}(\mathrm{Tot}(A(J_F\otimes J_G)))$, 
representing the same cup product.
\end{proof}

We now turn to the specific situation of theorem \ref{thm-cl-univ}(3), that is
$F=\Gamma^\ell$ and $G=\Gamma^m$. We first determine explicit maps between
the bicomplexes $A(J_\ell \otimes J_m)$ and $A(J_{\ell+m})$, which
lift the multiplication $\Gamma^\ell(gl^{(1)})\otimes\Gamma^m(gl^{(1)})
\to \Gamma^{\ell+m}(gl^{(1)})$ and
the diagonal $\Gamma^{\ell+m}(gl^{(1)}) \to \Gamma^\ell(gl^{(1)})
\otimes\Gamma^m(gl^{(1)})$. To do this,
we first need new information about the twist compatible coresolutions 
$J_\ell$ from \cite[Prop 3.21]{Touze}. 
\begin{Lemma}\label{lm-info-compl-res}
Let $\ell,m$ be positive integers. 
\begin{enumerate}
\item The multiplication $\Gamma^\ell\otimes \Gamma^m\to \Gamma^{\ell+m}$ lifts to a twist compatible chain map 
$J_\ell\otimes J_m\to J_{\ell+m}$. This chain map is given in degree $0$ by the shuffle product 
$(\otimes^\ell)\otimes(\otimes^m)= J^0_\ell\otimes J^0_m\to J^0_{\ell+m}=\otimes^{\ell+m}$ which sends a tensor 
$\otimes_{i=1}^{m+\ell} x_i$ to the sum $\sum_{\sigma\in Sh(\ell,m)}\otimes_{i=1}^{m+\ell} x_{\sigma^{-1}(i)}$.
\item The diagonal $\Gamma^{\ell+m}\to \Gamma^\ell\otimes \Gamma^m $ lifts to a twist compatible chain map 
$J_{\ell+m}\to J_\ell\otimes J_m$. This chain map equals the identity map in degree $0$.
\end{enumerate}
\end{Lemma}
\begin{proof} The reduced bar construction yields a functor from the category of Commutative Differential Graded Augmented algebras 
over $k$ to the category of Commutative Differential Graded bialgebras over $k$, see \cite{ML}, resp. \cite{FHT}, for the algebra, 
resp.\ coalgebra, structure (this bialgebra structure is actually a Hopf algebra structure but we don't need this fact).
 
The category of strict polynomial functors splits as a direct sum of
subcategories of homogeneous functors. Taking the $(m+\ell)$ polynomial degree part of
the multiplication (resp. comultiplication) of 
$\overline{B}(\overline{B}(S^*(-)))$ we obtain chain maps $\bigoplus J_i^\bullet \otimes J_j^\bullet\to J_{\ell+m}^\bullet$ and 
$ J_{\ell+m}^\bullet\to \bigoplus J_i^\bullet \otimes J_j^\bullet$ (the sums
are taken over all nonnegative integers $i,j$ such that $i+j=\ell+m$). The
bialgebra structure of $\overline{B}(\overline{B}(S^*(-)))$ is defined using
only the \emph{algebra} structure of $S^*$. But the multiplication of $S^*$
is a twist compatible map and the twist compatible category is additive and
stable under tensor products \cite[Lemma 3.8 and 3.10]{Touze}. So the chain
maps are twist compatible. 

Next, we identify the chain maps in degree
$0$. We begin with the map $J_\ell^0\otimes J_m^0\to J_{\ell+m}^0$ 
induced by the multiplication of the bar
construction. By 
\cite[Lemma 3.18]{Touze}, for all $i\ge 1$ we have $J_i^0=\overline{B}_1(S^*(-))^{\otimes i}=\otimes^i$. The product $\overline{B}(\overline{B}(S^*(-)))^{\otimes 2}\to \overline{B}(\overline{B}(S^*(-)))$ is given by the shuffle product 
formula \cite[p.313]{ML}, more precisely it sends the tensor $\otimes_{i=1}^{m+\ell} x_i$ to the sum $\sum_{\sigma\in Sh(\ell,m)}\otimes_{i=1}^{m+\ell} x_{\sigma^{-1}(i)}$. The signs in this shuffle product are all positive since the $x_i$ are 
elements of degree $1+1=2$ in the chain complex 
$\overline{B}_\bullet(\overline{B}(S^*(-)))$. 
The identification of the map $J_{m+\ell}^0\to J_m^0\otimes J_\ell^0$ induced 
by the diagonal is simpler. The coproduct in  
$\overline{B}(\overline{B}(S^*(-)))$ is given by the 
deconcatenation formula \cite[p.268]{FHT}:
$\Delta[x_1|\dots|x_{\ell+m}]=\sum_{i=0}^{m+\ell}[x_1|\dots|x_{i}]\otimes [x_{i+1}|\dots|x_{m+\ell}]$. Thus, the map 
$J_{m+\ell}^0\to J_m^0\otimes J_\ell^0$ sends the tensor product $\otimes_{i=1}^{m+\ell} x_i$ to itself.

Finally, with the description of the chain maps in degree $0$, one easily
checks that $J_\ell\otimes J_m\to J_{\ell+m}$, resp. $J_{\ell+m}\to
J_\ell\otimes J_m$, lifts the multiplication $\Gamma^\ell\otimes\Gamma^m\to
\Gamma^{\ell+m}$, resp. the comultiplication $\Gamma^{\ell+m}
\to \Gamma^\ell\otimes\Gamma^m$. (In fact, this actually proves that the
quasi isomorphism 
$\Gamma^*\to \overline{B}(\overline{B}(S^*(-)))$ is a Hopf algebra
morphism)
\end{proof}

Applying the functor $A$, we obtain:
\begin{Lemma}\label{lm-map-bicompl-induced}
Let $\ell,m$ be positive integers. 
\begin{enumerate}
\item The multiplication $\Gamma^\ell(gl^{(1)})\otimes \Gamma^m(gl^{(1)})
\to \Gamma^{\ell+m}(gl^{(1)})$ lifts to a map of bicomplexes 
$A(J_\ell\otimes J_m)\to A(J_{\ell+m})$. The restriction of this map to the columns of index $0$ equals
$$A(J_\ell\otimes J_m)^{0,\bullet}=(A_1^{\otimes \ell+m})_{[1]}\xrightarrow[]{sh_{[1]}} 
(A_1^{\otimes \ell+m})_{[1]}= A(J_{\ell+m})^{0,\bullet}$$ 
where $sh$ is the unsigned shuffle map, which sends a tensor $\otimes_{i=1}^{m+\ell} x_i$ to the sum 
$\sum_{\sigma\in Sh(\ell,m)}\otimes_{i=1}^{m+\ell} x_{\sigma^{-1}(i)}$.
\item The diagonal $\Gamma^{\ell+m}(gl^{(1)})\to \Gamma^\ell(gl^{(1)})
\otimes \Gamma^m(gl^{(1)}) $ lifts to a twist compatible chain map 
$A(J_{\ell+m})\to A(J_\ell\otimes J_m)$. The restriction of this map to the columns of index $0$ equals the identity map of 
$(A_1^{\otimes \ell+m})_{[1]}$.
\end{enumerate}
\end{Lemma}

Next we identify cycles representing
the cup products $\psi_\ell(x)\cup \psi_m(y)$ in the bicomplex
$H^0_\mathcal{P}(A(J_\ell\otimes J_m))$.
\begin{Lemma}\label{lm-repres-cup-bicompl}
Let $x\in \Gamma^\ell H^*_\mathcal{P}(gl^{(1)})$ and $y\in \Gamma^m H^*_\mathcal{P}(gl^{(1)})$ be classes of homogeneous degrees 
$2d$ and $2e$.
Then $\vartheta_{\ell+m}(x\otimes y)$ is a cocycle of bidegree $(0,2d)$
in the bicomplex $H^0_\mathcal{P}(A(J_\ell\otimes J_m))$. Moreover, it
represents the cup product
$\psi_\ell(x)\cup \psi_m(y)\in H^*_\mathcal{P}(\Gamma^\ell(gl^{(1)})\otimes \Gamma^m(gl^{(1)}))$.
\end{Lemma} 
\begin{proof}By definition, $\psi_\ell(x)$ is represented 
by the homogeneous cocycle $\vartheta_\ell(x)$ of bidegree $(0,2d)$ in the bicomplex $\Hom(\Gamma^{p\ell}(gl),A(J_\ell))$ (and similarily for $\psi_m(y)$). 
Then, by lemma \ref{lm-cup-prod-bicompl}, 
$\psi_\ell(x)\cup \psi_m(y)$ is represented by the cocycle $(\vartheta_\ell(x)\otimes \vartheta_m(y))\circ \Delta_{\ell p, mp}$ 
in the bicomplex $\Hom(\Gamma^{p(\ell+m)}(gl),A(J_\ell\otimes J_m))$. 
Now if $x=\otimes_{s=1}^\ell(e(i_s))$ and $y=\otimes_{s=\ell+1}^m(e(i_s))$,
 we compute that  $\vartheta_{\ell+m}(x\otimes y)$ and 
$(\vartheta_\ell(x)\otimes \vartheta_m(y))\circ \Delta_{\ell p, mp}$ both
equal the element $(\otimes_{i=1}^{\ell+m}z_i)\circ\Delta_{p,\dots,p}$, where
$\Delta_{p,\dots,p}$ is the 
diagonal $\Gamma^{p(\ell+m)}(gl)\to \Gamma^p(gl)^{\otimes \ell+m}$.
\end{proof}

We are now ready to prove theorem \ref{thm-cl-univ}(3).
We begin with the commutativity of the diagram involving the multiplication. 
Let $x\in \Gamma^\ell H^*_\mathcal{P}(gl^{(1)})$ and $y\in \Gamma^m
H^*_\mathcal{P}(gl^{(1)})$ be homogeneous elements of respective degrees 
$2d$ and $2e$. By lemmas  \ref{lm-map-bicompl-induced} and 
\ref{lm-repres-cup-bicompl},  
$m_{\ell,m\,*}(\psi_\ell(x)\cup \psi_m(y))$ is represented by the cocycle 
$$\sum_{\sigma\in Sh(\ell,m)}\nolimits \sigma.\vartheta_{\ell+m}(x\otimes y)$$
of bidegree $(0, 2d+2e)$ in the bicomplex $H^0_\mathcal{P}(A(J_{\ell+m}))$. 
By definition of 
$\psi_{\ell+m}$, $\psi_{\ell+m}(m_{\ell,m}(x\otimes y))$ is 
represented by the cocycle 
$$\vartheta_{\ell+m}\left(\sum_{\sigma\in Sh(\ell,m)}\nolimits \sigma.(x\otimes y)\right) $$
in the same bicomplex. Since $\vartheta_{\ell+m}$ is equivariant (lemma
\ref{lm-vartheta}), these two 
cocycles are equal. 
Hence, the diagram involving the multiplication is commutative.
The diagram involving the comultiplication commutes for a similar reason: if $x\in \Gamma^{\ell+m} 
H^*_\mathcal{P}(gl^{(1)})$, the cohomology classes 
$(\psi_\ell\cup \psi_m)(\Delta_{\ell,m}(x))$ and $\Delta_{\ell,m\,*}(\psi_{\ell+m}(x))$ are both represented by the 
cycle $\vartheta_{\ell+m}(x)$. This concludes the proof of theorem \ref{thm-cl-univ}(3).

\subsection{Proof of theorem \ref{thm-cl-univ}(2)}

To prove theorem \ref{thm-cl-univ}(2), it suffices to prove for all $n\ge p$ the injectivity of the composite:
\begin{align*}\Gamma^\ell H^*_{\mathcal{P}}(gl^{(1)})\xrightarrow[]{\psi_\ell} H^*_{\mathcal{P}}(\Gamma^\ell(gl^{(1)}))\xrightarrow[]{\phi_{\Gamma^\ell(gl^{(1)}),n}} 
&H^*(GL_{n,k},\Gamma^\ell(\mathfrak{gl}_n^{(1)}))\\&\xrightarrow[]{\Delta_{1,\dots,1\,*}} 
H^*(GL_{n,k},\mathfrak{gl}_n^{(1)\,\otimes \ell})\;.
\end{align*}
By naturality of the maps $\phi_{\Gamma^\ell(gl^{(1)}),n}$ \cite[Th 1.3]{Touze} and by the compatibility of the $\phi_i$ with 
diagonals and cup products given in theorem \ref{thm-cl-univ}(3), this composite equals the composite:
\begin{align*}\Gamma^\ell H^*_{\mathcal{P}}(gl^{(1)})\hookrightarrow H^*_{\mathcal{P}}(gl^{(1)})^{\otimes \ell}
\to H^*(GL_{n,k},\mathfrak{gl}_n^{(1)})^{\otimes\ell}\xrightarrow[]{\cup} H^*(GL_{n,k},\mathfrak{gl}_n^{(1)\,\otimes \ell}) 
\end{align*} 
Thus, the proof of theorem \ref{thm-cl-univ}(2) follows from:
\begin{Lemma} Let $k$ be a field of characteristic $p>0$ and let $j\ge 1$ be an integer. For all 
$n\ge p$, the following map is injective:
$$\begin{array}[t]{cccc}
\cup_{i=1}^j \phi_{gl^{(1)},n}:& H^*_{\mathcal{P}}(gl^{(1)})^{\otimes j}  &\to &  
H^*(GL_{n,k},(\mathfrak{gl}_n^{(1)})^{\otimes j})\\
& \otimes_{i=1}^j c_i&\mapsto & \cup_{i=1}^j \phi_{gl^{(1)},n}(c_i)
\end{array}.
$$  
\end{Lemma}
\begin{proof} We prove this lemma by reducing our cohomology classes to an infinitesimal one parameter subgroup 
$\Ga_{1}$ of $GL_{n,k}$, as it is done in \cite{SFB}. 
Since $n\ge p$, we can find a $p$-nilpotent matrix $\alpha\in\mathfrak{gl}_n$. Using this matrix, we define an embedding 
$\Ga_1\to \Ga\xrightarrow[]{exp_\alpha} GL_{n,k}$. For all $\ell$, this embedding makes the $GL_{n,k}$-module 
$\mathfrak{gl}_n^{(1)\otimes \ell}$ into a trivial $\Ga_1$-module. 
Thus there is an isomorphism
$H^*(\Ga_1,\mathfrak{gl}_n^{(1)\otimes \ell})\simeq H^*(\Ga_1,k)\otimes \mathfrak{gl}_n^{(1)\otimes \ell}$.   
The algebra $H^*(\Ga_{1},k)$ is computed in \cite{CPSvdK}. In particular, $H^{even}(\Ga_{1},k)=k[x_1]$ is a polynomial 
algebra on one generator $x_1$ of degree $2$. Let's thrash out the compatibility of this isomorphism with the cup product. 
If $x_1^{\ell}\otimes \beta_\ell$ and $x_1^{m}\otimes \beta_m$ are classes in $H^*(\Ga_1,k)\otimes 
\mathfrak{gl}_n^{(1)\otimes \ell}$, resp. $H^*(\Ga_1,k)\otimes \mathfrak{gl}_n^{(1)\otimes m}$, their cup product is the class $x_1^{\ell+m}\otimes (\beta_\ell\otimes\beta_m)$ in $H^*(\Ga_1,k)\otimes \mathfrak{gl}_n^{(1)\otimes \ell+m}$.

We recall that $H^*_{\mathcal{P}}(gl^{(1)})$ is a graded module 
with basis the classes $e_1(i)$ of degree $2i$ for $0\le i< p$.
By \cite[Th 4.9]{SFB}, the composite
$$H^*_{\mathcal{P}}(gl^{(1)})\xrightarrow[\simeq]{\phi_{gl^{(1)},n}} H^*(GL_{n,k},\mathfrak{gl}_n^{(1)})\to H^*(\Ga_1,\mathfrak{gl}_n^{(1)})\simeq H^*(\Ga_1,k)\otimes \mathfrak{gl}_n^{(1)} $$ 
sends $e_1(i)$ to the class $x_1^i\otimes (\alpha^{(1)})^i\in H^*(\Ga_1,k)\otimes \mathfrak{gl}_n^{(1)}$. Since restriction to 
$\Ga_{1}$ is compatible with cup products, the composite
$$H^*_{\mathcal{P}}(gl^{(1)})^{\otimes j} \to   H^*(GL_{n,k},\mathfrak{gl}_n^{(1)\otimes j}) \to H^*(\Ga_1,k)\otimes \mathfrak{gl}_n^{(1)\otimes j} $$
sends the tensor product $\otimes_{\ell=1}^j e(i_\ell)$ to the class $x_1^{\sum_{\ell=1}^j i_\ell}\otimes (\otimes_{\ell=1}^j (\alpha^{(1)})^{i_\ell})$. 
As a result, this composite sends the basis $\left(\otimes_{i=\ell}^j e(i_\ell)\right)_{(i_1,\dots,i_\ell)}$ into the linearly 
independent
 family $\left(x_1^m \otimes (\otimes_{\ell=1}^j (\alpha^{(1)})^{i_\ell})\right)_{(m,i_1,\dots,i_\ell)}$.
Hence the map  $\cup_{i=1}^j \phi_{gl^{(1)},n}$ is injective.
\end{proof}

\part{The second proof}

\section{The starting point}
Many notations are as in \cite{cohGrosshans}. 
 We work over a field $k$ of positive characteristic $p$. Fix an integer $n$, with $n>1$.
 The important case is when $n$ is large. So if one finds this convenient, one may take $n\geq p$.
We wish to draw conclusions from the following result.
\begin{Theorem}[Lifted universal cohomology classes \cite{Touze}]\label{divided powers}
There are cohomology classes $c[m]$
so that
\begin{enumerate}
\item $c[1]\in H^{2}(\GL_{n,k},\gl_n^{(1)})$ is nonzero,
\item For $m\geq1$ the class $c[m]\in H^{2m}(\GL_{n,k},\Gamma^{m}(\gl_n^{(1)}))$ lifts $c[1]\cup\cdots\cup c[1]\in H^{2m}(\GL_{n,k},\bigotimes^{m}(\gl_n^{(1)}))$.
\end{enumerate}
\end{Theorem}
\begin{Remark}We really have in mind that $c[1]$ is the Witt vector class of \cite[section 4]{cohGrosshans},
which is certainly nonzero. 
The computation of $H^{2}(\GL_{n,k},\gl_n^{(1)})$ is
easy, using \cite[Corollary (3.2)]{CPSvdK}.
One finds that $H^{2}(\GL_{n,k},\gl_n^{(1)})$ is  one dimensional. Thus any nonzero $c[1]$ is up to scaling
equal to the Witt vector class. 
\end{Remark}
\section{Using the classes}
We write $G$ for  $\GL_{n,k}$, the algebraic group $\GL_n$ over $k$. Sometimes it is instructive to restrict to $\SL_n$ or other reductive
subgroups of $\GL_n$. We leave this to the the reader.

\subsection{Other universal classes}
We recall some constructions from \cite{cohGrosshans}.
If $M$ is a finite dimensional vector space over $k$ and
$r\geq1$, we have a natural
homomorphism between symmetric algebras $S^*(M^{\#(r)})\to S^{*}(M^{\#(1)})$
induced by the
map $M^{\#(r)}\to S^{p^{r-1}}(M^{\#(1)})$ which raises an element to the power
$p^{r-1}$. It is a map of bialgebras.
Dually we have the bialgebra map
$\pi^{r-1}:\Gamma^{*}(M^{(1)})\to \Gamma^{*}(M^{(r)})$ whose kernel is the
ideal generated by $\Gamma^1(M^{(1)})$ through $\Gamma^{p^{r-1}-1}(M^{(1)})$.
So $\pi^{r-1}$ maps $\Gamma^{p^{r-1}a}(M^{(1)})$ onto $\Gamma^{a}(M^{(r)})$.

\begin{Notation} We now introduce analogues of the classes $e_r$ and
$e_r^{(j)}$ of Friedlander and Suslin \cite[Theorem 1.2,
Remark 1.2.2]{Friedlander-Suslin}.
We write $\pi^{r-1}_*(c[ap^{r-1}])\in
H^{2ap^{r-1}}(G,\Gamma^{a}(\gl_n^{(r)}))$
as $c_r[a]$. Next we get
$c_r[a]^{(j)}\in H^{2ap^{r-1}}(G,\Gamma^{a}(\gl_n^{(r+j)}))$
by Frobenius twist.
As in \cite{Friedlander-Suslin} a notation like
$S^*(M(i))$ means the symmetric algebra $S^*(M)$, but
graded, with $M$ placed in degree $i$.
\end{Notation}

Here is the analogue of \cite[Lemma 4.7]{cohGrosshans}.

\begin{Lemma}\label{cup}
The $c_i[a]^{(r-i)}$ enjoy the following properties ($r\geq i\geq1$)
\begin{enumerate}
\item \label{cuphom} There is a homomorphism of graded algebras
$S^*(\gl_n^{\#(r)}(2p^{i-1}))\to H^{2p^{i-1}*}(G_{r},k)$
given on $\gl_n^{\#(r)}(2p^{i-1})=H^0(G_{r},\gl_n^{\#(r)})$
by cup product with the restriction of
$c_i[1]^{(r-i)}$ to $G_r$. If $i=1$, then  it
is given on $S^a(\gl_n^{\#(r)}(2))=H^0(G_{r},S^a(\gl_n^{\#(r)}))$
by cup product with the restriction of
$c[a]^{(r-1)}$ to $G_r$.
\item \label{new e_r} For $r\geq1$ the restriction of
$c_r[1]$ to $ H^{2p^{r-1}}(G_{1},\gl_n^{(r)})$ is nontrivial,
so that $c_r[1]$ may serve as the universal class $e_r$ in
\cite[Thm 1.2]{Friedlander-Suslin}.
\end{enumerate}
\end{Lemma}

\paragraph{Proof}
\ref{cuphom}.
When $M$ is a $G$-module, one has a commutative diagram
$$\begin{array}{ccc}
\Gamma^m M\otimes \bigotimes^mM^\# & {\rightarrow} &\bigotimes  ^m M\otimes \bigotimes^mM^\#\\[.3em]
\qquad \downarrow& & \downarrow\\[.3em]
\Gamma^m M\otimes S^mM^\# & \rightarrow& k\\
\end{array}$$
Take $M=\gl_n^{(1)}$.
There is a homomorphism of algebras $\bigotimes^*(\gl_n^{\#(1)})\to H^{2*}(G_{1},k)$
given on $\gl_n^{\#(1)}$ by cup product with $c[1]$. (We do not mention obvious restrictions to
subgroups like  $G_1$
any more.)
 On $\bigotimes^m(\gl_n^{\#(1)})$
it is given by cup product with $c[1]\cup\cdots\cup c[1]$, so by Theorem \ref{divided powers}
it is also given by cup product with $c[m]$, using the pairing $\Gamma^m M\otimes \bigotimes^mM^\#\to k$. 
As this pairing factors through $\Gamma^m M\otimes S^mM^\#$ we get that the induced algebra map
$S^*(\gl_n^{\#(1)})\to H^{2*}(G_{1},k)$ is  given by cup product with $c[m]$ on $S^m(\gl_n^{\#(1)})$.
If we compose with the algebra map $S^*(\gl_n^{\#(i)})\to S^{p^{i-1}*}(\gl_n^{\#(1)})$ we
get an algebra map $\psi:S^*(\gl_n^{\#(i)})\to H^{2p^{i-1}*}(G_{1},k)$ given on $\gl_n^{\#(i)}$
by cup product with $c[p^{i-1}]$, using the pairing $\gl_n^{\#(i)}\otimes \Gamma^{p^{i-1}}(\gl_n^{(1)})\to k$. 
This pairing 
factors through 
$\id\otimes\pi^{i-1}:\gl_n^{\#(i)}\otimes \Gamma^{p^{i-1}}(\gl_n^{(1)})\to \gl_n^{\#(i)}\otimes 
\gl_n^{(i)}$, so the homomorphism $\psi$ is given on $\gl_n^{\#(i)}$ by cup product
with $\pi^{i-1}_*c[p^{i-1}]=c_i[1]$. We can lift it to an algebra map $S^*(\gl_n^{\#(i)})\to H^{2p^{i-1}*}(G_{i},k)$
by simply still using cup product with $c_i[1]$ on $\gl_n^{\#(i)}$.
Pull back along the $(r-i)$-th Frobenius homomorphism $G_r\to G_i$ 
and you get an algebra map $\psi^{(r-i)}:S^*(\gl_n^{\#(r)})\to H^{2p^{i-1}*}(G_{r},k)$, given on
 $\gl_n^{\#(r)}$ by cup product with $c_i[1]^{(r-i)}$. If $i=1$, pull back the cup product with $c[m]$ on $S^m(\gl_n^{\#(1)})$
 to a cup product with  $c[m]^{(r-1)}$ on $S^m(\gl_n^{\#(r)}(2))$. This then describes the homomorphism $\psi^{(r-1)}$
 degree-wise.
\\
\ref{new e_r}. In fact if we restrict $c_r[1]$ as in remark
\cite[4.1]{cohGrosshans} to
$H^{2p^{r-1}}(\Ga_{1},(\gl_n^{(r)})_{p^r\alpha})=
H^{2p^{r-1}}(\Ga_{1},k)\otimes(\gl_n^{(r)})_{p^r\alpha}$, then even
that restriction is nontrivial. That is because the Witt
vector class generates the polynomial ring $H^\even(\Ga_{1},k)$, see
\cite[I 4.26]{Jantzen}. And at this level $\Gamma^m\hookrightarrow \bigotimes^m$ gives an isomorphism,
showing that $c[m]$ restricts to the $m$-th power of the polynomial generator.
\qed\\


\subsection{Noetherian homomorphisms}
Let $A$ be a commutative $k$-algebra. The cohomology algebra  $H^*(G,A)$
is then graded commutative, so we must also consider graded commutative algebras.
\begin{Definition}
If $f:A\to B$ is a homomorphism of graded commutative $k$-algebras, we call $f$ \emph{noetherian} if
$f$ makes $B$ into a noetherian left $A$ module.
\end{Definition}

\begin{Remark}
In algebraic geometry a noetherian homomorphism between finitely generated  commutative $k$-algebras
is called a \emph{finite} morphism.
With our terminology we wish to stress the importance of chain conditions in our arguments.
\end{Remark}

\begin{Lemma}
The composite of noetherian homomorphisms is noetherian.
\end{Lemma}
\paragraph{Proof} If $A\to B$ and $B\to C$ are noetherian, view $C$ as a quotient of
the module $B^r$ for some $r$.
\qed

\begin{Lemma}\label{second noeth}
If the composite of $A\to B$ and $B\to C$ is noetherian, so is $B\to C$.
\end{Lemma}
\paragraph{Proof} View $B$-submodules of $C$ as $A$-modules.
\qed

\begin{Remark}\label{map}
In this lemma $A\to C$ and $B\to C$ must be homomorphisms, but $A\to B$ could be just a map.
\end{Remark}

\begin{Lemma}\label{even int}
Suppose $B$ is finitely generated as a graded commutative $k$-algebra. Then $f:A\to B$ is noetherian
if and only if $B^\even$ is integral over $f(A^\even)$.
\end{Lemma}

\paragraph{Proof}
The map $B^\even\to B$ is noetherian. So if $B^\even$ is integral over $f(A^\even)$, then 
$f$ is noetherian. Conversely, if $f$ is noetherian and $b\in B^\even$, then for some
$r$ one must have $b^r\in \sum_{i<r}f(A)b^i$. But then in fact $b^r\in \sum_{i<r}f(A^\even)b^i$.
\qed\\

In particular one has

\begin{Lemma}
Suppose $B$ is a finitely generated  commutative $k$-algebra. Let $n>1$ and let $A$ be a subalgebra of $B$
containing $x^n$ for
every $x\in B$. Then $A\hookrightarrow B$ is noetherian and $A$ is also finitely generated.
\end{Lemma}
\paragraph{Proof} We follow Emmy Noether \cite{noether}. 
Indeed $B$ is integral over $A$. Take finitely many generators $b_i$ of $B$ and let
$C$ be the subalgebra generated by the $b_j^n$. Then $A$ is a $C$-submodule of $B$,
hence finitely generated.\qed 
\\

Invariant theory tells
\begin{Lemma}\cite[Thm. 16.9]{Grosshans book}\label{inv noeth}
Let $f:A\to B$ be a noetherian homomorphism of finitely generated graded commutative $k$-algebras with rational $G$ action. Then
$A^G\to B^G$ is noetherian.\hfill$\Box$
\end{Lemma}

\begin{Lemma}
Let $f:A\to B$ be a noetherian homomorphism of finitely generated graded commutative $k$-algebras with rational $G_r$ action. Then
$H^*(G_r,A)\to H^*(G_r,B)$ is noetherian.
\end{Lemma}
\paragraph{Proof}
Take $C=H^0(G_r,A^\even )$ or its subalgebra generated by the $p^r$-th powers in $A^\even$.
Then apply \cite[Theorem 1.5, Remark 1.5.1]{Friedlander-Suslin}.
\qed\\

We will need a minor variation on a theorem of Friedlander and Suslin.
\begin{Theorem}\label{smaller height}
Let $r\geq1$.
Let $S\subset G_r$ be an infinitesimal group scheme over $k$ of height at most $r$.
Let further $C$ be a finitely generated commutative $k$-algebra (considered as a trivial $S$-module)
and let $M$ be a noetherian $C$-module on which $S$ acts by  $C$-linear transformations.
Then $H^*(S,M)$ is a noetherian module over the algebra $\bigotimes_{i=1}^rS^*((\gl_n^{(r)})^\#(2p^{i-1}))\otimes C$,
with the map given as suggested by Lemma \ref{cup}.
\end{Theorem}

\begin{Corollary}\label{noeth res}
The restriction map $H^*(G_r,C)\to H^*(G_{r-1},C)$ is noetherian.
\end{Corollary}
\paragraph{Proof}
Take $S=G_{r-1}$ and note that the map $\bigotimes_{i=1}^rS^*((\gl_n^{(r)})^\#(2p^{i-1}))\otimes C\to H^*(G_{r-1},C)$
factors through $H^*(G_r,C)$.
\qed\\

\paragraph{Proof of the theorem}
The key difference with 
\cite[Theorem 1.5, Remark 1.5.1]{Friedlander-Suslin} is that we do not require the height of $S$ to be $r$.
(As $S\subset G_r$ the fact that its height is at most $r$ is automatic.)
Thus, to start their inductive argument, we must also check the obvious case where $r=1$ and $S$ is the trivial group.
The rest of the proof goes through without change. \qed\\

\begin{Remark}
If $S$ has height $s$, then the map $(\gl_n^{(r)})^\#(2p^{i-1})\to H^{2p^{i-1}}(S,k)$ is trivial for $r-i\geq s$.
\end{Remark}

\subsection{Cup products on the cochain level}
As we will need a differential graded algebra structure on Hochschild--Serre spectral sequences, we now expand the 
discussion of the Hochschild complex in \cite[I 5.14]{Jantzen}.
Let  $L$ be an affine algebraic  group scheme over the field $k$, $N$ a normal subgroup scheme, $R$ a commutative $k$-algebra on which 
 $L$ acts rationally by algebra automorphisms.
We have a Hochschild complex $C^*(L,R)$ with $R\otimes k[L]^{\otimes i}$ in degree $i$.
Define a cup product on $C^*(L,R)$ as follows. If $u\in C^r(L,R)$ and $v\in C^s(L,R)$, then $u\cup v$
is defined in simplified notation by 
$$(u\cup v)(g_1,\ldots g_{r+s})=u(g_1,\ldots,g_r).{}^{g_1\cdots g_r} v(g_{r+1},\ldots ,g_{r+s}),$$
where ${}^gr$ denotes the image of $r\in R$ under the action of $g$.
One checks that
\begin{Lemma}
With this cup product $C^*(L,R)$ is a differential graded algebra.
\end{Lemma}

In particular, taking for $R$ the algebra $k[L]$ with $L$ acting by right translation, we get the differential
graded algebra $C^*(L,k[L])$, quasi-isomorphic to $k$. And the action by left translation on $k[L]$
is by $L$-module isomorphisms, so
this makes $C^*(L,k[L])$ into a differential graded algebra with $L$ action.
It consists of injective $L$-modules in every degree.
We write $C^*(L)$  for this differential graded algebra with $L$ action. One has $C^i(L)=k[L]^{\otimes i+1}$ and 
this is our elaboration of \cite[I 4.15 (1)]{Jantzen}.
 
The Hochschild--Serre spectral sequence $$E_2^{rs}=H^r(L/N,H^s(N,R))\Rightarrow H^{r+s}(L,R)$$
can now be based on the double complex $( C^*(L/N)\otimes (C^*(L)\otimes R)^N )^{L/N}$. 
The tensor product over $k$ of two differential graded algebras is again a differential graded algebra
and the spectral sequence
inherits differential graded algebra 
structures \cite[3.9]{Benson II}
from such structures on $C^*(L)\otimes R$, $(C^*(L)\otimes R)^N$, $C^*(L/N)\otimes (C^*(L)\otimes R)^N$.

\subsection{Hitting invariant classes}
We now come to the main result of this section, which is the counterpart of \cite[Cor. 4.8]{cohGrosshans}.
It does not seem to follow from the cohomological finite generation conjecture, but we will show it implies the conjecture.
\begin{Theorem}\label{dominate}Let $r\geq1$. Further
let $A$ be a finitely generated
commutative $k$-algebra with $G$ action.
Then $H^\even(G,A)\to H^0(G,H^*(G_r,A))$ is  noetherian.
\end{Theorem}
\begin{Remark}
Recall that $H^0(G,H^*(G_r,A))$ is finitely generated as a $k$-algebra,
by \cite{Friedlander-Suslin} and invariant theory.
\end{Remark}
\paragraph{Proof}
\begin{list}{{\bf\stepcounter{enumi}\arabic{enumi}}.}{\topsep0pt
\parsep0pt\itemsep0pt
\setcounter{enumi}{0}\itemindent2em\labelwidth2em\leftmargin0pt}
\item If $M$ is a $G$-module on which $G_r$ acts trivially, then $H^0(G,M)$ and $H^0(G/G_r,M)$ 
  denote the same subspace of $M$. We may thus switch between these variants.
\item We argue by induction on $r$.
Put $C=H^0(G_r,A)$. Then $C$ contains the elements of $A$ raised to
the power $p^r$, so  $C$
is also a finitely generated algebra and $A$ is
a noetherian module over it. 
\item Let $r=1$. 
This case is the same as in \cite{cohGrosshans}.
By \cite[Thm 1.5]{Friedlander-Suslin}
$H^*(G_1,A)$ is a noetherian module over the finitely generated
algebra
$$R=S^*((\gl_n^{(1)})^\#(2))\otimes
C.$$
Then, by invariant theory \cite[Thm. 16.9]{Grosshans book},
$H^0(G,H^*(G_1,A))$ is a noetherian module over the finitely
generated algebra $H^0(G,R)$.
By lemma \ref{cup}
we may take the algebra homomorphism $R\to H^*(G_1,A)$ of
\cite{Friedlander-Suslin}
to be based on cup product with our $c[a]=c[a]^{(0)}$ on the summand
$S^a((\gl_n^{(1)})^\#(2))\otimes
C$.
But then the map
$H^0(G,R)\to H^*(G_1,A)$
factors, as  a linear map, through $H^\even(G,A)$.
This settles the case $r=1$ by \ref{map}.
\item
Now let the level $r$ be greater than $1$.
We are going to follow the analysis in \cite[section 1]{Friedlander-Suslin} to peel off
one level at a time. Heuristically, in the tensor product of Theorem \ref{smaller height}
we treat one factor at a time. That is the main difference with the argument in \cite{cohGrosshans}.

Thus, consider the Hochschild--Serre spectral sequence 
$E^{ij}_2(C)=H^i(G_r/G_{r-1},H^j(G_{r-1},C))\Rightarrow H^{i+j}(G_r,C)$.
We first wish to argue that this spectral sequence stops, meaning that
$E^{**}_s(C)=E^{**}_\infty(C)$ for some finite~$s$. This is proved in \cite[section 1]{Friedlander-Suslin}
for a very similar spectral sequence. So we imitate the argument. We need to apply \cite[Lemma 1.6]{Friedlander-Suslin}
and its proof. We use $H^\even (G_r,C)$ for the $A$ of that lemma and  $S^*((\gl_n^{(r)})^\#(2))$ for its $B$.
We map $H^\even (G_r,C)$ in the obvious way to the abutment $H^* (G_r,C)$ and for $B\to E^{*0}_2(k)=H^*(G_r/G_{r-1},k)=
H^*(G_{1},k)^{(r-1)}$ we use the $(r-1)$-st Frobenius twist of the map $S^*((\gl_n^{(1)})^\#(2))\to H^*(G_{1},k)$ of
Lemma \ref{cup}. So we use the class $c[a]^{(r-1)}$ on $S^a((\gl_n^{(r)})^\#(2))$.
By Corollary \ref{noeth res} and Lemma \ref{even int} the restriction map $H^\even (G_r,C)\to H^*(G_{r-1},C)$ is noetherian and by
Theorem \ref{smaller height}, compare \cite[Cor 1.8]{Friedlander-Suslin}, 
it follows that the $H^\even (G_r,C)\otimes S^*((\gl_n^{(r)})^\#(2))$ module
$E^{**}_2(C)=H^*(G_{1},H^*(G_{r-1},C)^{(1-r)})^{(r-1)}$ is noetherian, so the spectral sequence stops, say at $E^{**}_s(C)$.
Note also that the image of $H^\even (G_r,C)\otimes S^*((\gl_n^{(r)})^\#(2))$ in $E^{**}_2(C)$ consists of permanent
cycles.
\item
As the spectral sequence is one of graded commutative
differential graded algebras, the $p$-th power of an even cochain in a page passes to the next page.
As the spectral sequence stops at page $s$ one finds that for an $x\in E^{\even,\even}_2(C)$ the
power $x^{p^s}$ is a permanent cycle. Let $P$ be the algebra generated by permanent cycles in $E^{\even,\even}_2(C)$.
Then $P\to E^{ij}_t(C)$ is noetherian for $2\leq t\leq\infty$. So $P^G\to (E^{**}_\infty(C))^G$ is
noetherian by Lemma \ref{inv noeth}.
\item
By the inductive assumption $H^0(G,H^*(G_{r-1},C\otimes S^*((\gl_n^{(r)})^\#(2))))$ 
is noetherian over  $H^\even(G, C\otimes S^*((\gl_n^{(r)})^\#(2)))$.
By step 1 we may
rewrite $H^0(G,H^j(G_{r-1},C\otimes S^i((\gl_n^{(r)})^\#(2))))$ as $H^0(G/G_{r-1},H^j(G_{r-1},C\otimes S^i((\gl_n^{(r)})^\#(2))))$. The latter description will be needed in the sequel.
We may map $H^0(G/G_{r-1},H^j(G_{r-1},C\otimes S^i((\gl_n^{(r)})^\#(2))))$  by 
restriction
to $H^0(G_r/G_{r-1},H^j(G_{r-1},C)\otimes S^i((\gl_n^{(r)})^\#(2)))$ and then to 
$E^{2i,j}_2(C)=H^{2i}(G_r/G_{r-1},H^{j}(G_{r-1},C))$ by cup product
with $c[i]^{(r-1)}$. So we now have a map from $H^\even(G, C\otimes S^*((\gl_n^{(r)})^\#(2)))$ to $E^{**}_2(C)$.
We will factor it further.
\item One checks that the map from $H^\even(G, C\otimes S^*((\gl_n^{(r)})^\#(2)))$ to
$H^0(G_r/G_{r-1},H^*(G_{r-1},C)\otimes S^*((\gl_n^{(r)})^\#(2)))$ of step~6 factors naturally through the algebra
$H^\even (G_r,C)\otimes S^*((\gl_n^{(r)})^\#(2))$ of step 4. 
Moreover, as the algebra $H^\even (G_r,C)\otimes S^*((\gl_n^{(r)})^\#(2))$ acts on the full spectral sequence, we may make 
$H^\even(G, C\otimes S^*((\gl_n^{(r)})^\#(2)))$ act on the full spectral sequence by way of that algebra.
\item
The noetherian map $H^\even (G_r,C)\otimes S^*((\gl_n^{(r)})^\#(2))\to
E^{**}_2(C)$ factors through $H^0(G_r/G_{r-1},H^*(G_{r-1},C\otimes S^*((\gl_n^{(r)})^\#(2))))$.
But then 
$H^0(G_r/G_{r-1},H^*(G_{r-1},C\otimes S^*((\gl_n^{(r)})^\#(2))))\to E^{**}_2(C)$ is noetherian by 
Lemma \ref{second noeth}. So by Lemma \ref{inv noeth} the map 
$H^0(G/G_{r-1},H^*(G_{r-1},C\otimes S^*((\gl_n^{(r)})^\#(2))))\to (E^{**}_2(C))^G$ is noetherian.
Combining with the inductive hypothesis,
 we learn that $H^\even(G, C\otimes S^*((\gl_n^{(r)})^\#(2)))\to (E^{**}_2(C))^G$ is noetherian.
It lands  in $P^G$, because the map in step 4 lands in $P$.
We conclude that $H^\even(G, C\otimes S^*((\gl_n^{(r)})^\#(2)))\to (E^{**}_\infty(C))^G$ is noetherian.
\item
We filter $H^\even(G, C\otimes S^*((\gl_n^{(r)})^\#(2)))$ by putting $H^\even(G, C\otimes S^t((\gl_n^{(r)})^\#(2)))$
in $H^\even(G, C\otimes S^*((\gl_n^{(r)})^\#(2)))^{\geq j}$ for $t\geq j$.
As in \cite[section~1]{Friedlander-Suslin} the filtered algebra may be identified with its associated graded,
and the map $H^\even(G, C\otimes S^*((\gl_n^{(r)})^\#(2)))\to H^*(G_r,C)$ respects filtrations.
Now we care about $H^*(G_r,C)^G$ as a module for $H^\even(G, C\otimes S^*((\gl_n^{(r)})^\#(2)))$.
To see that it is noetherian we may pass as in \cite[section 1]{Friedlander-Suslin}
to the associated graded, where one puts $(H^*(G_r,C)^G)^{\geq j}=(H^*(G_r,C)^G)\cap H^*(G_r,C)^{\geq j}$.
This associated graded of $H^*(G_r,C)^G$ is a submodule of $(E^{**}_\infty(C))^G$, containing the image of 
$H^\even(G, C\otimes S^*((\gl_n^{(r)})^\#(2)))$. So it is indeed noetherian.
We conclude that $H^*(G_r,C)^G$ is noetherian over $H^\even(G, C\otimes S^*((\gl_n^{(r)})^\#(2)))$.
\item
As in the case $r=1$ the map $H^\even(G, C\otimes S^*((\gl_n^{(r)})^\#(2)))\to H^*(G_r,C)^G$
factors, as a linear map, through $H^\even(G, C)$, so $H^\even(G, C)\to H^*(G_r,C)^G$ is noetherian by \ref{map}.
As $H^*(G_r,C)^G\to H^*(G_r,A)^G$ is noetherian, the result follows. 
\end{list}\qed\\

\subsection{Cohomological finite generation}
Now let $A$ be a finitely generated
commutative $k$-algebra with $G$ action.
We wish to show that $H^*(G,A)$ is finitely generated, following the same path as in \cite{cohGrosshans},
but using improvements from \cite{Srinivas vdK}.
As in \cite{cohGrosshans} we denote by $\A$ the coordinate ring of a flat family with general fiber $A$
and special fiber $\gr A$ \cite[Theorem 13]{Grosshans contr}. Choosing $r$ as in \cite[Prop. 3.8]{cohGrosshans}
we have the spectral sequence
$$E_2^{ij}=H^i(G/G_r,H^j(G_r,\gr A))\Rightarrow H^{i+j}(G,\gr A).$$ 
and
$R=H^*(G_r,\gr A)^{(-r)}$ is a finite module
over the algebra $$\bigotimes_{a=1}^rS^*((\gl_n)^\#(2p^{a-1}))\otimes
\hull(\gr A).$$ This algebra has a good filtration, and by the main result of \cite{Srinivas vdK}
the ring $R$ has finite good filtration dimension. In particular, there are only finitely many 
nonzero $H^i(G,R)$. Thus the same main result tells that $E_2^{0*}\to E_2^{**}$ is noetherian. 
Now $H^0(G/G_r,H^*(G_r,\A))\to H^0(G/G_r,H^*(G_r,\gr A))$ is noetherian by \cite{Friedlander-Suslin}
and lemma \ref{inv noeth}.
And $H^*(G,\A))\to H^0(G/G_r,H^*(G_r,\A))$ is noetherian by theorem \ref{dominate}, so another application of 
\cite[Lemma 1.6]{Friedlander-Suslin} (with $B=k$)
shows that $H^*(G,\A)\to H^*(G,\gr A)$ is noetherian. 


There is a map of spectral sequences from a totally degenerate spectral sequence 
$$E(\A):\quad E_1^{ij}(\A)=H^{i+j}(G,\gr_{-i}\A)\Rightarrow H^{i+j}(G, \A),$$
with pages $H^*(G,\A)$, to the spectral sequence $$E(A):\quad E_1^{ij}(A)=H^{i+j}(G,\gr_{-i}A)\Rightarrow H^{i+j}(G, A).$$
This makes that $H^*(G,\A)$ acts on $E(A)$ and 
the noetherian homomorphism  $H^*(G,\A)\to H^*(G,\gr A)$ is used in standard fashion \cite[slogan 3.9]{reductive}
to make the spectral sequence $E(A)$ stop.
It follows easily that $H^*(G,A)$ is finitely generated.
So far $G$ was $\GL_{n,k}$.
As explained in some detail in \cite{reductive} this case
implies our Cohomological Finite Generation Conjecture (over fields.)
\begin{Remark}
The spectral sequence $E(A)$ is based on filtering the Hochschild complex of $A$. 
As it lives in the second quadrant,
the exposition of multiplicative structure in \cite[3.9]{Benson II} does not apply as stated. 
(In order to avoid convergence issues \cite{Benson II} uses a filtration that reaches a maximum.)
But \cite[Ch XV, Ex. 2]{Cartan-Eilenberg} is sufficiently general to cover our case. Or see \cite{Massey}.
\end{Remark}

\end{document}